\documentclass[12pt,leqno, draft]{article}
\usepackage{amsfonts}
\usepackage{amsmath, amsthm, amsfonts, amssymb, color}
\usepackage{mathrsfs}
\usepackage{color}
\usepackage{stmaryrd}
\makeatletter
\newcommand{\Spvek}[2][r]{%
  \gdef\@VORNE{1}
  \left(\hskip-\arraycolsep%
    \begin{array}{#1}\vekSp@lten{#2}\end{array}%
  \hskip-\arraycolsep\right)}

\def\vekSp@lten#1{\xvekSp@lten#1;vekL@stLine;}
\def\vekL@stLine{vekL@stLine}
\def\xvekSp@lten#1;{\def\temp{#1}%
  \ifx\temp\vekL@stLine
  \else
    \ifnum\@VORNE=1\gdef\@VORNE{0}
    \else\@arraycr\fi%
    #1%
    \expandafter\xvekSp@lten
  \fi}
\makeatother

\newtheorem{thm}{Theorem}[section]

\newtheorem{lem}[thm]{Lemma}

\newtheorem{rem}[thm]{Remark}
\theoremstyle{definition}
\newtheorem{defn}{Definition}[section]
\newcommand{\scr}[1]{\mathscr #1}
\definecolor{wco}{rgb}{0.5,0.2,0.3}

\def\beq{\begin{equation}}
\def\neqq{\end{equation}}
\def\benu{\begin{enumerate}}
\def\nenu{\end{enumerate}}

\numberwithin{equation}{section} \theoremstyle{remark}

\newcommand{\ua}{\uparrow}

\title{{\bf    Large Deviations Principle for SDEs with Dini Continuous Drifts} \footnote{Supported in
 part by  NNSFC (11801406).} }
\author{
{\bf  Lingyan Cheng$^{a)}$,  Xing Huang$^{a)}$}\\
\footnotesize{$^{a)}$Center for Applied Mathematics, Tianjin
University, Tianjin 300072, China}\\
\footnotesize{ xinghuang@tju.edu.cn}\\
\footnotesize{chengly@amss.ac.cn}}

\date{}
\begin{document}
\allowdisplaybreaks
\def\R{\mathbb R}  \def\ff{\frac} \def\ss{\sqrt} \def\B{\mathbf
B}
\def\N{\mathbb N} \def\kk{\kappa} \def\m{{\bf m}}
\def\ee{\varepsilon}\def\ddd{D^*}
\def\dd{\delta} \def\DD{\Delta} \def\vv{\varepsilon} \def\rr{\rho}
\def\<{\langle} \def\>{\rangle} \def\GG{\Gamma} \def\gg{\gamma}
  \def\nn{\nabla} \def\pp{\partial} \def\E{\mathbb E}
\def\d{\text{\rm{d}}} \def\bb{\beta} \def\aa{\alpha} \def\D{\scr D}
  \def\si{\sigma} \def\ess{\text{\rm{ess}}}
\def\beg{\begin} \def\beq{\begin{equation}}  \def\F{\scr F}
\def\Ric{\text{\rm{Ric}}} \def\Hess{\text{\rm{Hess}}}
\def\e{\text{\rm{e}}} \def\ua{\underline a} \def\OO{\Omega}  \def\oo{\omega}
 \def\tt{\tilde} \def\Ric{\text{\rm{Ric}}}
\def\cut{\text{\rm{cut}}} \def\P{\mathbb P} \def\ifn{I_n(f^{\bigotimes n})}
\def\C{\scr C}   \def\G{\scr G}   \def\aaa{\mathbf{r}}     \def\r{r}
\def\gap{\text{\rm{gap}}} \def\prr{\pi_{{\bf m},\varrho}}  \def\r{\mathbf r}
\def\Z{\mathbb Z} \def\vrr{\varrho} \def\ll{\lambda}
\def\L{\scr L}\def\Tt{\tt} \def\TT{\tt}\def\II{\mathbb I}
\def\i{{\rm in}}\def\Sect{{\rm Sect}}  \def\H{\mathbb H}
\def\M{\scr M}\def\Q{\mathbb Q} \def\texto{\text{o}} \def\LL{\Lambda}
\def\Rank{{\rm Rank}} \def\B{\scr B} \def\i{{\rm i}} \def\HR{\hat{\R}^d}
\def\to{\rightarrow}\def\l{\ell}\def\iint{\int}
\def\EE{\scr E}\def\no{\nonumber}
\def\A{\scr A}\def\V{\mathbb V}\def\osc{{\rm osc}}
\def\BB{\scr B}\def\Ent{{\rm Ent}}\def\3{\triangle}
\def\U{\scr U}\def\8{\infty}\def\1{\lesssim}\def\HH{\mathrm{H}}
 \def\T{\scr T}
\def\benu{\begin{enumerate}}
\def\nenu{\end{enumerate}}

\maketitle

\begin{abstract}
In this paper, using Zvonkin type transform, the large deviation principle is proved for stochastic differential equations with Dini continuous drifts, where the existed methods for large deviation principle are unavailable. The method and result are new in related fields.  Moreover, the result is also extended to a class of  degenerate stochastic differential equations with Dini continuous drifts.

\end{abstract} \noindent
 {\bf AMS subject Classification:}\ 60H35, 60H10, 60C30  \\
\noindent {\bf Keywords:} Dini continuity, Zvonkin type transform, Large deviation, Degenerate SDEs
 \vskip 2cm

\section{Introduction}
The large deviation principle (LDP for short) is proved for various stochastic differential equations (SDEs) with Lipschitz continuous drift. For instance,
Freidlin and Wentzell \cite{FW12} firstly studied the LDP for the finite dimensional setting,
where the SDE is driven by finitely many Brownian motions and its coefficients satisfy suitable regularity properties. Peszat \cite{P94} (also the references therein) investigated the LDP for stochastic partial differential equations (SPDEs) under global Lipschitz condition on the nonlinear term. Cerrai and R\"ockner \cite{CR04} obtained the LDP for  stochastic reaction-diffusion systems with multiplicative noise under local Lipschitz conditions. Moreover,  the LDP for semilinear parabolic equations on a Gelfand triple was proved by Chow in \cite{C92}. R\"ockner, Wang and Wu \cite{RWW06} established the LDP for stochastic porous media equations within the variational framework. All these papers mainly used the classical ideas of discretization approximations and the contraction principle, which was firstly developed by Freidlin and Wentzell.

Budhiraja, Dupuis and Maroulas \cite{BDM} also get the LDP of the infinite dimensional setting by the weak convergence method (see \cite{BD}). This approach is now a powerful tool which has been extensively used to prove LDP for various stochastic dynamical systems. For instance, Cerrai and Freidlin \cite{CF} established the LDP for the langevin equation, see also \cite{ BD13,  L10, YR, MSS, RZ, RZZ10,  RS,  ZX17, Z17} and the references therein for more works. There are also some results with non-Lipschitz coefficients, for instance, \cite{FZ, KS0,Lan}.

Recently, pathwise uniqueness of SDEs/SPDEs with singular drifts are proved. The main idea is to construct Zvonkin's transform (\cite{AZ}) which is a homeomorphism map to transform the original SDEs to a new one, where the singular drift is killed and the pathwise uniqueness can be obtained. This technique strongly depends on the regularity of the solution to PDE like \eqref{PDE} below with singular coefficients.  Wang \cite{W} proved the pathwise uniqueness for semi-linear SPDEs with Dini continuous drift and non-degenerate noise. In \cite{WZ}, Wang and Zhang studied existence and uniqueness for stochastic Hamiltonian system with  H\"older-Dini continuous drifts, where the noise is degenerate. There are also many other results on this topic, see \cite{Fe-Fl,GM,KR,CDR,Z} and references therein.

So far, there are no results on LDP for SDEs with singular drifts, where the existed methods, either discretization approximations or weak convergence are unavailable. The aim of this paper is to solve this problem. To this end, we need to search for new technique and Zvonkin's transform offers an effective method to regularized the singular drifts. The idea is to use Zvonkin's transform to change the SDEs with singular drifts as a new one with Lipschitz continuous coefficients, where the LDP holds. Then we can obtain the LDP for the original SDE by the inverse of Zvonkin's transform and the definition of LDP.

Throughout the paper, the following notations will be used. For $T>0$, $d\in\mathbb{N}^+$, let $C([0,T],\mathbb{R}^d)$ be all $\mathbb{R}^d$-valued and continuous functions on $[0,T]$. For a function $f$ from $\mathbb{R}^m$ to $\mathbb{R}^n$, set $\|f\|_{\infty}:=\sup_{x\in \R^m}|f(x)|$.

Before moving on, let us recall some knowledge on LDP.
 \begin{defn} \label{rate} Let $S$ be a Polish space. A function $I:S\to \mathbb{R}^1$ is called a rate function, if for any constant $c>0$, the level set $\{f; I(f)\le c\}$ is compact in $S$.
 \end{defn}
 \begin{defn}\label{FW} Let $S$ be a Polish space. We call a family of $S$-valued random variable $\{Z^\varepsilon\}_{\varepsilon\in(0,1)}$ satisfies LDP with speed function $\varepsilon^{-1}$ and rate function $I:S\to[0,\infty)$, if the following conditions hold.
\begin{itemize}
\item[(1)] For any closed subset $F\subset S$,
$$
\limsup_{\vv\rightarrow0^+}\vv\log\P(Z^{\vv}\in F)\le -\inf_{f\in F}I(f).
$$
\item[(2)] For any open subset $G\subset S$,
$$
\liminf_{\vv\rightarrow0^+}\vv\log\P(Z^{\vv}\in G)\ge -\inf_{f\in G}I(f).
$$
\end{itemize}
\end{defn}
From now on, we fix $T>0$. Next, we give an existed result Lemma \ref{LG} from \cite{FW12} which will be used in the sequel, see also the introduction in \cite{GW}. Consider SDE on $\mathbb{R}^n$:
\beq\label{b2=0}
\d \tilde{X}_t^\vv=b_1^{\varepsilon}(\tilde{X}_t^\vv)+\sqrt{\varepsilon}\sigma(\tilde{X}_t^\vv)\d
W_t, ~~~t\in [0,T],~~~\tilde{X}_0^\vv=x_0\in\mathbb{R}^n,
\end{equation}
where $\varepsilon\in(0,1)$, $b_1^{\varepsilon}: \R^n\to\R^n$, $\si: \R^n\to\R^n\otimes
\R^n$, and $(W_t)_{t\in [0,T]}$ is an $n$-dimensional Brownian motion defined
on a complete filtration probability space $(\OO, \F, (\F_t)_{t\in [0,T]},
\P)$. Without loss of generality, we assume $x_0=0$.
\begin{enumerate}
\item[{\bf(A1)}] There exists a constant $L>0$ such that for any $\varepsilon\in(0,1)$,
\begin{align}\label{Lip}\|\sigma(x)-\sigma(y)\|+|b_{1}^{\varepsilon}(x)-b_{1}^{\varepsilon}(y)|\leq L|x-y|,\ \ x,y\in\mathbb{R}^n.
\end{align}
Moreover, there exists a Lipschitz continuous function $b_1^0:\mathbb{R}^n\to\mathbb{R}^n$ such that
\begin{align}\label{lim}\lim_{\varepsilon\to0}\left\{\sup_{x\in\mathbb{R}^n}|b_{1}^{\varepsilon}(x)-b_{1}^{0}(x)|\right\}=0.
\end{align}
\end{enumerate}
Let $C([0,T],\R^n)$ be equipped with sup-norm, and define rate function $I:C([0,T],\R^n)\to[0,\infty)$ as
\beq\label{rf}
I(f) = \frac{1}{2} \inf_{f= g(h), h\in \mathcal{H}} \|h\|_H^2,\ \ f\in C([0,T],\R^n),
\neqq
where
$$
\mathcal{H} =\left\{ h \in C([0,T],\R^n); \|h\|_H^2:=\int_0^T |\dot{h}_t|^2 dt < \infty \right\}
$$
and for any $h\in \mathcal{H}$, $g(h)\in C([0,T],\R^n)$ satisfies
\begin{align}\label{gh}
(g(h))_t=\int_0^t b_1^0((g(h))_s)\d s +\int_0^t \sigma((g(h))_s)\dot{h}_s\d s, \quad t \in [0,T].
\end{align}
\begin{rem} Under {\bf(A1)}, for any $\varepsilon\in(0,1)$, \eqref{b2=0} has a uniqueness strong solution denoted by $\{(\tilde{X}_t^\vv)_{t\in[0,T]} \}$. Furthermore, {\bf(A1)} also implies that for any $h\in\mathcal{H}$, $g(h)$ defined above is the uniqueness solution to the following deterministic differential equation:
\begin{align}\label{gh0}
\d Z_t= b_1^0(Z_t)\d t +\sigma(Z_t)\dot{h}_t\d t, \quad t \in [0,T], Z_0=0.
\end{align}
\end{rem}
\beg{lem}\label{LG}
Under {\bf (A1)}, the family $\{(\tilde{X}_t^\vv)_{t\in[0,T]} \}_{\vv \in(0,1)}$ obeys an LDP on $C([0,T]; \R^n)$ with the speed function $\vv^{-1}$ and the rate function $I$ given by \eqref{rf}.
\end{lem}
The outline of this paper is organized as follows: In Section 2, we study the LDP for non-degenerate SDEs with singular drift;  In Section 3, we investigate LDP for degenerate SDEs with singular drift.
\section{LDP for Non-degenerate SDEs}
In this section, we add a small singular interruption in \eqref{b2=0}, i.e. consider the following SDE on $\R^n$:
 \beq\label{1.1} \d X_t^\vv=b_1^{\varepsilon}(X_t^\vv)+\varepsilon b_2(X_t^\vv)\d t+\sqrt{\varepsilon}\sigma(X_t^\vv)\d
W_t, ~~~t\in [0,T],~~~X_0^\vv=x_0,
\end{equation}
where $\varepsilon,\si,b_1^{\varepsilon} $ and $(W_t)_{t\in [0,T]}$  are introduced in Section 1, and $b_2: \R^n\to\R^n$ is the singular drift. Without loss of generality, we assume $x_0=0$.

To characterize the singularity of $b_2$, we introduce some definitions which comes from \cite{BGT} and \cite{WZ}.
\beg{defn}\label{D1.1}
\benu
\item[(1)]
An increasing function $\phi: [0,\infty)\to[0,\infty)$ is called a Dini function if
\begin{equation*}
\int_0^1{\frac{\phi(s)}{s}\d s}<\infty.
\end{equation*}
\item[(2)]
A function $f$ defined on the Euclidean space is called Dini continuous if
\begin{equation*}
|f(x)-f(y)|\leq \phi(|x-y|)
\end{equation*}
holds for some Dini function $\phi$.
\item[(3)]
A measurable function $\phi: [0,\infty)\to[0,\infty)$ is called a $slowly$ $varying$ function at zero (see \cite{BGT}) if for any $\delta>0$,
\begin{equation*}
\lim_{t\to0}\frac{\phi(\delta t)}{\phi(t)}=1.
\end{equation*}
\nenu
\end{defn}
Let $\D_{0}$ be the set of all Dini functions, and $\T_{0}$ the set of all slowly varying functions at zero
that are bounded away from $0$ and $\infty$ on $[\varepsilon,\infty)$ for any $\varepsilon>0$. Notice that the typical examples for functions contained in $\D_{0}\cap\T_{0}$ are $\phi(t):= (\log(1+t^{-1}))^{-\beta}$ for $\beta>1$.

To obtain the LDP for \eqref{1.1}, we make the following assumptions.
 \beg{enumerate}
\item[\bf{(A1')}] Besides {\bf (A1)}, there exists a constant $K>1$ such that $$\sup_{\varepsilon\in(0,1)}\|b_1^\varepsilon\|_{\infty}+\|b_2\|_{\infty}\leq K$$ and
\begin{equation}\label{bao3}K^{-1}I\leq \sigma\sigma^{\ast}\leq K I.
\end{equation}
\item[\bf{(A2)}] There exists $\phi\in\D_0\cap\T_0$ such that
\begin{equation}\label{phi}
|b_2(x)-b_2(y)|\leq \phi(|x-y|),\ \ x,y\in\R^n.
\end{equation}
\end{enumerate}
Under {\bf(A1')} and {\bf(A2)}, \eqref{1.1} admits a unique
non-explosive strong solution $(X_t^\vv)_{t\in[0,T]}$; see, e.g.,
\cite[Corollary 1.5]{WZ}. In fact, by Zvonkin's transform, we can kill $b_2$, see \eqref{Yv} below for more details.

Our main result is
\beg{thm}\label{T1.1}
Assume {\bf (A1')}-{\bf (A2)}, then $\{(X_t^\vv)_{t\in[0,T]} \}_{\vv \in(0,1)}$ obeys LDP on $C([0,T]; \R^n)$ with the speed function $\vv^{-1}$ and the rate function $I$ given by \eqref{rf}.
\end{thm}
\begin{rem} Due to the singularity of $b_2$, we need to give stronger condition {\bf(A1')} in Theorem  \ref{T1.1} than {\bf(A1)} in Lemma \ref{LG}, see the proof of Theorem \ref{T1.1} for more details.
\end{rem}
\subsection{Proof of Theorem \ref{T1.1}}
In order to obtain LDP for \eqref{1.1}, we adopt Zvonkin type transform to change \eqref{1.1} to a new equation with Lipschitz continuous coefficients, where the Freidlin-Wentzell's theorem (\cite{FW12}) can be available. Let $(e_i)_{i\ge1}$ be an  orthogonal basis of $\R^n.$ For any
$\lambda
>0$, consider the following $\R^n$-valued PDE:
\beq\label{PDE}\beg{split}
\scr L u_\ll+b_2+\nabla_{b_2}u_\ll=\lambda
u_\ll,
\end{split}\end{equation}
where
\beg{equation*}
\scr L:= \frac{1}{2}\sum^n_{i,j=1}{\langle
(\si\si^*)e_{i},e_{j}\rangle}\nabla_{e_{i}}\nabla_{e_{j}}.
 \end{equation*}

By \cite[Theorem 3.10]{WZ} with $d_1=0$, $d_2=n$, there exists a constant $\lambda_0>0$ such that for any $\lambda\geq \lambda_0$, the equation \eqref{PDE} has a unique solution $u_{\lambda}$ satisfying
\begin{equation}\label{u}
\|u_\lambda\|_{\infty}+\|\nn u_\lambda\|_{\infty}+\|\nabla^2 u_\lambda\|_{\infty}\le \frac{1}{2}.
\end{equation}


For any $\lambda\geq\lambda_0$, let $\theta_\lambda:\mathbb{R}^n\to\mathbb{R}^n$ be defined by $\theta_\lambda (x):=x+u_\ll(x),x\in\R^n$. By \eqref{u}, $\theta_\lambda$ is a homeomorphism on $\mathbb{R}^n$. Let $\theta_\lambda^{-1}$ denote the inverse of $\theta_\lambda$, then it holds that $\nabla \theta_\lambda^{-1}=(\nabla \theta_\lambda)^{-1}$.

We now in a position to complete the Proof of Theorem \ref{T1.1}.

\smallskip
\begin{proof}[\bf Proof of Theorem \ref{T1.1}] Throughout the whole
proof, we assume $ \lambda\geq \lambda_0$. Since
 \beq\label{Xv} \d X^\varepsilon_t=b_{1}^{\varepsilon}(X^\varepsilon_t)+\varepsilon b_{2}(X^\varepsilon_t)\d t+\sqrt{\varepsilon}\sigma(X^\varepsilon_t)\d
W_t, \ \ t\in[0,T],~~~X_0^\vv=x_0,
\end{equation}
applying It\^{o}'s formula to $\theta_\lambda (X_t^\vv)$, we deduce from \eqref{PDE} that
\begin{equation}\label{b5}
\beg{split}
\d \theta_\lambda (X_t^\vv)=\varepsilon\lambda u_\ll(X_t^\vv)\d t+(\nabla
\theta^\ll b_{1}^{\varepsilon})(X_t^\vv)\d t+\sqrt{\varepsilon}(\nabla\theta^\ll\si)(X_t^\vv)\d W_t,\ \ t\in[0,T].
\end{split}
\end{equation}
Denote $Y_t^\vv := \theta_\lambda (X_t^\vv)$, then \eqref{b5} becomes
\begin{equation}\label{Yv}
\begin{aligned}
\d Y_t^\vv
=&\varepsilon\lambda u_\ll(\theta_\lambda^{-1}(Y_t^\vv))\d t+(\nabla
\theta_\ll b_{1}^{\varepsilon})(\theta_\lambda^{-1}(Y_t^\vv))\d t+\sqrt{\varepsilon}(\nabla\theta_\ll\si)(\theta_\lambda^{-1}(Y_t^\vv))\d W_t\\
=&:\tilde{b}^\vv(Y_t^\vv)\d t+\sqrt{\vv} \tilde{\sigma}(Y_t^\vv)\d W_t,\quad t\in[0,T],\quad Y_0^\vv=\theta_\lambda (x_0),
\end{aligned}
\end{equation}
where
$$
\tilde{b}^\vv(x)=\vv\lambda u_\ll(\theta_\lambda^{-1}(x))+(\nabla
\theta_\ll b_{1}^{\vv})(\theta_\lambda^{-1}(x)), \ \ \tilde{\sigma}(x)=(\nabla\theta_\ll\si)(\theta_\lambda^{-1}(x)),\ \ x\in\mathbb{R}^n.
$$

Since $\theta_\lambda$ is a diffeomorphic operator, by {\bf(A1')} and \eqref{u}, $\tilde{b}^\vv$ and $\tilde{\sigma}$ satisfy the following conditions:
\benu
\item[(1)]
for some constant $\tilde{K}>1$, we have
$$
\|\tilde{\sigma}(x)-\tilde{\sigma}(y)\|+|\tilde{b}^\varepsilon(x)-\tilde{b}^\varepsilon(y)| \le \tilde{K}|x-y|,\ \ x,y \in \R^n.
$$
\item[(2)] Let $\tilde{b}^0:= (\nabla \theta_\ll b_{1}^{0})\circ\theta_\lambda^{-1}$, then
$$\lim_{\varepsilon\to0}\|\tilde{b}^{\varepsilon}-\tilde{b}^{0}\|_\infty=0.$$

\nenu
By Lemma \ref{LG},
$\{Y_t^\vv, t \in [0,T]\}_{\varepsilon\in(0,1)}$ satisfies the LDP in $C([0,T],\R^n)$ with the speed function $\vv^{-1}$ and the good rate function given by
\beq\label{rate function}
I^Y(f):= \frac{1}{2} \inf_{f= g^Y(h), h\in \mathcal{H}} \|h\|_H^2
\neqq
with
$$
(g^Y(h))_t=\int_0^t \tilde{b}^0((g^Y(h))_s)\d s +\int_0^t \tilde{\sigma}((g^Y(h))_s)\dot{h}_s\d s, \quad t \in [0,T].
$$
This implies that
\begin{itemize}
\item[(i)] for any constant $c>0$, the level set $\{f; I^Y(f)\le c\}$ is compact in $C([0,T];\R^n)$;
\item[(ii)]for any closed subset $F\subset C([0,T];\R^n)$,
$$
\limsup_{\vv\rightarrow0^+}\vv\log\P(Y^{\vv}\in F)\le -\inf_{f\in F}I^Y(f);
$$
\item[(iii)] for any open subset $G\subset C([0,T];\R^n)$,
$$
\liminf_{\vv\rightarrow0^+}\vv\log\P(Y^{\vv}\in G)\ge -\inf_{f\in G}I^Y(f).
$$
\end{itemize}
Define
\beq\label{rf1}
I^X(f):= \frac{1}{2} \inf_{f= g^X(h), h\in \mathcal{H}} \|h\|_H^2
\neqq
with
$$
(g^X(h))_t=\int_0^t b_1^0((g^X(h))_s)\d s +\int_0^t \sigma((g^X(h))_s)\dot{h}_s\d s, \quad t \in [0,T].
$$
In the following, we will prove that
$\{X_t^\vv, t \in [0,T]\}_{\varepsilon\in(0,1)}$ satisfies the LDP in $C([0,T],\R^n)$ with the speed function $\vv^{-1}$ and the good rate function $I^X$.
This will be completed in Lemma \ref{LDP}.
\end{proof}
\begin{lem}\label{LDP} Assume {\bf (A1')} and {\bf (A2)}, then $\{X_t^\vv, t \in [0,T]\}_{\varepsilon\in(0,1)}$ satisfies the LDP in $C([0,T],\R^n)$ with the speed function $\vv^{-1}$ and the good rate function $I^X$.
\end{lem}
\begin{proof}We only need to prove that (i)-(iii) hold with $Y$ replaced by $X$.
For any $\lambda\geq\lambda_0$, define $\Theta_\lambda$ on $C([0,T];\R^n)$ as
$$(\Theta_\lambda(\xi))_t=\theta_\lambda(\xi_t),\ \ t\in[0,T],\xi\in C([0,T];\R^n).$$
Then it is not difficult to see that $\Theta_\lambda$ is a homeomorphism on $C([0,T];\R^n)$. In fact, for any $\xi\in C([0,T];\R^n)$,
$$|(\Theta_\lambda(\xi))_t-(\Theta_\lambda(\xi))_s|=\theta_\lambda(\xi_t)-\theta_\lambda(\xi_s)\leq \|\nabla\theta_\lambda\|_\infty|\xi_t-\xi_s|,$$
which means $\Theta_\lambda(\xi)\in C([0,T];\R^n)$. Moreover, for any $\xi\in C([0,T];\R^n)$, let $\eta\in C([0,T];\R^n)$ be defined as $\eta_s=\theta^{-1}_\lambda(\xi_s),\ s\in[0,T]$. Then $\Theta_\lambda(\eta)=\xi$. On the other hand, for any $\xi,\bar{\xi}\in C([0,T];\R^n)$ satisfying $\Theta_\lambda(\xi)=\Theta_\lambda(\bar{\xi})$, i.e.,  $\theta_\lambda(\xi_s)=\theta_\lambda(\bar{\xi}_s), s\in[0,T]$, we have $\xi=\bar{\xi}$. So, $\Theta_\lambda$ is a bijection on $C([0,T];\R^n)$. Moreover, for any $\xi,\tilde{\xi}\in C([0,T];\R^n)$, we have  $$\|\Theta_\lambda(\xi)-\Theta_\lambda(\tilde{\xi})\|_{\infty}= \sup_{s\in[0,T]}|\theta_\lambda(\xi_s)-\theta_\lambda(\tilde{\xi}_s)|\leq \|\nabla\theta_\lambda\|_\infty\sup_{t\in[0,T]}|\xi_t-\tilde{\xi}_t| =\|\nabla\theta_\lambda\|_\infty\|\xi-\tilde{\xi}\|_{\infty},$$
which means that $\Theta_\lambda$ is a continuous map. Similarly, $\Theta_\lambda^{-1}$ is also a continuous map. Thus, $\Theta_\lambda$ is a homeomorphism.

(i) We firstly prove that $I^X$ is a rate function.  $I^X=I^{Y}(\Theta_\lambda(\cdot))$. By chain rule, we have
\begin{align*}
\theta_\lambda ((g^X(h))_t)&=\int_0^t [(\nabla \theta_\lambda b_1^0)\circ\theta_\lambda^{-1}](\theta_\lambda((g^X(h))_s))\d s\\
 &\ \ +\int_0^t [(\nabla\theta_\lambda\sigma)\circ\theta_\lambda^{-1}](\theta_\lambda((g^X(h))_s))\dot{h}_s\d s\\
 &=\int_0^t \tilde{b}^0(\theta_\lambda((g^X(h))_s))\d s\\
 &\ \ +\int_0^t \tilde{\sigma}(\theta_\lambda((g^X(h))_s))\dot{h}_s\d s, \quad t \in [0,T].
\end{align*}
By the uniqueness of solution, we have $\theta_\lambda ((g^X(h))_t)=(g^Y(h))_t, t\in[0,T]$, i.e. $\Theta_\lambda (g^X(h))=g^Y(h)$. Combining the definition of $I^X$ and $I^Y$, it is easy to see that $I^X=I^{Y}(\Theta_\lambda(\cdot))$. Thus, for any $c>0$, $\{f; I^X(f)\le c\}=\{f;I^{Y}(\Theta_\lambda(f))\leq c\}=\Theta_\lambda^{-1}\{f; I^Y(f)\le c\}$. Since $\{f; I^Y(f)\le c\}$ is a compact set, and $\Theta_\lambda$ is a homeomorphism, we conclude that $\{f; I^X(f)\le c\}$ is a compact set.

(ii) For any closed subset $F\subset C([0,T];\R^n)$,
\begin{align*}
&\limsup_{\vv\rightarrow0^+}\vv\log\P(X^{\vv}\in F)\\
&=\limsup_{\vv\rightarrow0^+}\vv\log\P(Y^{\vv}\in \Theta_\lambda(F))\\
&\le -\inf_{f\in \Theta_\lambda(F)}I^Y(f)\\
&= -\inf_{f\in F}I^Y(\Theta_\lambda(f))=-\inf_{f\in F}I^X(f).
\end{align*}
Similarly, for any open subset $G\subset C([0,T];\R^n)$,
$$
\liminf_{\vv\rightarrow0^+}\vv\log\P(X^{\vv}\in G)\ge -\inf_{f\in G}I^X(f).
$$
Thus, (iii) holds.

We finish the proof.
\end{proof}

\section{LDP for Degenerate SDEs}
Consider the following degenerate SDEs on $\mathbb{R}^{d_1+d_2}$:
\beq\label{E1}
\begin{cases}
\d X_t=\bar{b}^\varepsilon(X_t, Y_t)\d t, \\
\d Y_t=\bar{B}^\varepsilon(X_t, Y_t)\d t+\varepsilon b(Y_t)\d t+\sqrt{\varepsilon}\sigma(Y_t)\d W_t, \\
(X_{0},Y_0)=(x_0,y_0)\in\mathbb{R}^{d_1+d_2},
\end{cases}
\end{equation}
where $\varepsilon\in(0,1)$, $W=(W_t)_{t\geq 0}$ is an $d_2$-dimensional standard Brownian motion with respect to a complete filtration probability space $(\OO, \F, \{\F_{t}\}_{t\ge 0}, \P)$, $\bar{b}^\varepsilon:\mathbb{R}^{d_1+d_2}\to \mathbb{R}^{d_1}, \bar{B}^\varepsilon :\mathbb{R}^{d_1+d_2}\to \mathbb{R}^{d_2}, b: \mathbb{R}^{d_2}\to \mathbb{R}^{d_2}$
and $\sigma: \mathbb{R}^{d_2}\to \mathbb{R}^{d_2}\otimes\mathbb{R}^{d_2}$ are measurable and locally bounded (bounded on bounded sets). Again we assume $(x_0,y_0)=0$.

Suppose that  there exists $\phi\in\D_0\cap\T_0$ and a constant $K>1$ such that the following conditions hold.
 \beg{enumerate}
\item[\bf{(H1)}]  $\|\bar{B}^\varepsilon\|_{\infty}+\|b\|_{\infty}\leq K$,
\begin{align}\label{Lip'}\|\sigma(y_1)-\sigma(y_2)\|\leq K|y_1-y_2|,\ \ y_1,y_2\in\R^{d_2},
\end{align}
and
$$|\bar{b}^\varepsilon(z_1)-\bar{b}^\varepsilon(z_2)|+|\bar{B}^\varepsilon(z_1)-\bar{B}^\varepsilon(z_2)|\leq K|z_1-z_2|,\ \ z_1,z_2\in\mathbb{R}^{d_1+d_2}.$$
Moreover,
\begin{equation}\label{bao3'}
K^{-1}I_{d_2\times d_2}\leq \sigma\sigma^{\ast}\leq K I_{d_2\times d_2}.
\end{equation}
\item[\bf{(H2)}] There exist Lipschitz continuous functions $\bar{b}^0:\mathbb{R}^{d_1+d_2}\to \mathbb{R}^{d_1}$ and $\bar{B}^0:\mathbb{R}^{d_1+d_2}\to\mathbb{R}^{d_2}$ such that
\begin{align}\label{lim'}
\lim_{\varepsilon\to0}\left\{\|\bar{b}^{\varepsilon}-\bar{b}^{0}\|_\infty\right\}=0,
\end{align}
and
\begin{align}\label{lim''}
\lim_{\varepsilon\to0}\left\{\|\bar{B}^{\varepsilon}-\bar{B}^{0}\|_\infty\right\}=0.
\end{align}
\item[\bf{(H3)}] (Regularity of $b_2$ )
\begin{equation}\label{phi'}
|b(y_1)-b(y_2)|\leq \phi(|y_1-y_2|), \ \ y_1,y_2\in\R^{d_2}.
\end{equation}
\end{enumerate}
Under {\bf(H1)} and {\bf(H3)}, for any $\varepsilon\in(0,1)$, \eqref{E1} admits a unique non-explosive strong solution $(X_t^\vv, Y_t^\vv)_{t\in[0,T]}$; see, e.g.,
\cite[Theorem 1.1]{WZ}.

Let $C([0,T],\R^{d_2})$ be equipped with sup-norm, and define rate function $I:C([0,T],\R^{d_2})$ $\to[0,\infty)$ as
\beq\label{rf'}
I(f) = \frac{1}{2} \inf_{f= g(h), h\in \tilde{\mathcal{H}}} \|h\|_{\tilde{H}}^2,
\neqq
where
$$
\tilde{\mathcal{H}} =\left\{ h \in C([0,T],\R^{d_2}); \|h\|_{\tilde{H}}^2:=\int_0^T |\dot{h}_t|^2 dt < \infty \right\}
$$
and for any $h\in \tilde{\mathcal{H}}$, $g(h)\in C([0,T],\R^{d_1+d_2})$ satisfies
$$
(g(h))_t=\int_0^t (\bar{b}^0((g(h))_s), \bar{B}^0((g(h))_s)\d s +\int_0^t (0,\sigma((g(h))_s)\dot{h}_s)\d s, \quad t \in [0,T].
$$
\subsection{Main results}
The main result of this section is the following theorem.
\beg{thm}\label{T1.10} Assume {\bf{(H1)}}-{\bf{(H3)}}. The family $\{(X_t^\vv, Y_t^\vv))_{t\in[0,T]} \}_{\vv \in(0,1)}$ obeys the LDP on $C([0,T]; \R^{d_1+d_2})$ with the speed function $\vv^{-1}$ and the rate function $I$ given by \eqref{rf'}.
\end{thm}
\subsection{Proof of Theorem \ref{T1.10}}
Similarly to the proof of Theorem \ref{T1.1}, let $(e_i)_{i\ge1}$ be an  orthogonal basis of $\R^{d_2}.$ For any
$\lambda
>0$, consider the following $\R^{d_2}$-valued PDE:
\beq\label{PDE'}\beg{split}
\tilde{\L} u_\ll+b+\nabla_{b}u_\ll=\lambda
u_\ll,
\end{split}\end{equation}
where \beg{equation*}
\tilde{\L}:= \frac{1}{2}\sum^{d_2}_{i,j=1}{\langle
(\si\si^*)e_{i},e_{j}\rangle}\nabla_{e_{i}}\nabla_{e_{j}}.
 \end{equation*}

Then by \cite[Theorem 3.10]{WZ}, there exists a constant $\lambda_0>0$ such that for any $\lambda\geq \lambda_0$, the equation \eqref{PDE'} has a unique solution $u_{\lambda}$ satisfying
\begin{equation}\label{u'}
\|u_{\lambda}\|_{\infty}+\|\nn u_{\lambda}\|_{\infty}+\|\nabla^2 u_{\lambda}\|_{\infty}\le \frac{1}{2}.
\end{equation}


For any $\lambda\geq\lambda_0$, let $\theta_\lambda:\mathbb{R}^{d_2}\to\mathbb{R}^{d_2}$ be defined by $\theta_\lambda (x):=x+u_\ll(x),x\in\R^{d_2}$. By \eqref{u'}, $\theta_\lambda$ is a homeomorphism on $\mathbb{R}^{d_2}$. Let $\theta_\lambda^{-1}$ denote the inverse of $\theta_\lambda$, then it holds that $\nabla \theta_\lambda^{-1}=(\nabla \theta_\lambda)^{-1}$.
Throughout the whole
proof, we assume $ \lambda\geq \lambda_0$. Since
 \beq\label{Xv'} \begin{cases}
\d X^\varepsilon_t=\bar{b}^\varepsilon(X^\varepsilon_t, Y^\varepsilon_t)\d t, \\
\d Y^\varepsilon_t=\bar{B}^\varepsilon(X^\varepsilon_t, Y^\varepsilon_t)\d t+\varepsilon b(Y^\varepsilon_t)\d t+\sqrt{\varepsilon}\sigma(Y^\varepsilon_t)\d W_t, \\
(X_{0},Y_0)=(x_0,y_0)\in\mathbb{R}^{d_1+d_2},
\end{cases}
\end{equation}
it follows from It\^{o}'s formula and \eqref{PDE} that
\begin{equation}\label{b5'}\begin{cases}
\d X^\varepsilon_t=\bar{b}^\varepsilon(X^\varepsilon_t, Y^\varepsilon_t)\d t, \\
\d \theta_\lambda (Y^\varepsilon_t)=\varepsilon\lambda u_\ll(Y^\varepsilon_t)\d t+\nabla
\theta_\ll(Y^\varepsilon_t) \bar{B}^{\varepsilon}(X^\varepsilon_t, Y^\varepsilon_t)\d t+\sqrt{\varepsilon}(\nabla\theta_\ll\si)(Y^\varepsilon_t)\d W_t.
\end{cases}
\end{equation}
Denote $\tilde{Y}_t^\vv := \theta_\lambda (Y_t^\vv)$, then \eqref{b5'} can be written as
\begin{equation}\label{Yv'}
\begin{cases}
\d X^\varepsilon_t=\tilde{b}^\varepsilon(X^\varepsilon_t, \tilde{Y}^\varepsilon_t)\d t, \\
\d \tilde{Y}_t^\vv
=\tilde{B}^\vv(X_t^\vv,\tilde{Y}_t^\vv)\d t+\sqrt{\varepsilon}\tilde{\sigma}(\tilde{Y}_t^\vv)\d W_t,
\end{cases}
\end{equation}
where
$$
\tilde{B}^\vv(x,y)=\varepsilon\lambda u_\ll(\theta_\lambda^{-1}(y))+\nabla
\theta_\ll(\theta_\lambda^{-1}(y)) \bar{B}^{\varepsilon}(x,\theta_\lambda^{-1}(y)),$$
and
$$\tilde{b}^\vv(x,y)=\bar{b}^{\varepsilon}(x,\theta_\lambda^{-1}(y)),\ \ \tilde{\sigma}(y)=(\nabla\theta_\ll\si)(\theta_\lambda^{-1}(y)),\ \ (x,y)\in\mathbb{R}^{d_1+d_2}.
$$

Since $\theta_\lambda$ is a diffeomorphic operator, by {\bf(H1)}, {\bf(H2)} and \eqref{u'}, $\tilde{B}^\vv, \tilde{b}^\vv$ and $\tilde{\sigma}$ satisfy the following conditions:
\benu
\item[(1)]
There exists a constant $\tilde{K}>1$ such that for any $z_1=(x_1,y_1), z_2=(x_2,y_2)\in\R^{d_1+d_2}$,
$$
\|\tilde{\sigma}(y_1)-\tilde{\sigma}(y_2)\|+|\tilde{b}^\varepsilon(x_1,y_1)-\tilde{b}^\varepsilon(x_2,y_2)| +|\tilde{B}^\varepsilon(x_1,y_1)-\tilde{B}^\varepsilon(x_2,y_2)|\le \tilde{K}|z_1-z_2|.
$$
\item[(2)] Let $\tilde{b}^0(x,y)= \bar{b}^{0}(x,\theta_\lambda^{-1}(y))$ and $ \tilde{B}^0(x,y) := \nabla
\theta_\ll(\theta_\lambda^{-1}(y)) \bar{B}^{0}(x,\theta_\lambda^{-1}(y))$, $(x,y)\in\R^{d_1+d_2}$, then it holds that
\begin{align*}
\lim_{\varepsilon\to0}\left\{\|\tilde{b}^{\varepsilon}-\tilde{b}^{0}\|_\infty\right\}=0,
\end{align*}
and
\begin{align*}
\lim_{\varepsilon\to0}\left\{\|\tilde{B}^{\varepsilon}-\tilde{B}^{0}\|_\infty\right\}=0.
\end{align*}
\nenu
Again by Lemma \ref{LG},
$\{(X^\vv_t,\tilde{Y}_t^\vv), t \in [0,T]\}_{\vv\in(0,1)}$ satisfies the LDP in $C([0,T],\R^{d_1+d_2})$  with the speed function $\vv^{-1}$ and the good rate function $\tilde{I}$ given by
\beq\label{rate0}
\tilde{I}(f):= \frac{1}{2} \inf_{f= \tilde{g}(h), h\in \tilde{\mathcal{H}}} \|h\|_{\tilde{H}}^2
\neqq
with
$$
(\tilde{g}(h))_t=\int_0^t (\tilde{b}^0((\tilde{g}(h))_s),\tilde{B}^0((\tilde{g}(h))_s))\d s +\int_0^t (0,\tilde{\sigma}((\tilde{g}(h))_s)\dot{h}_s\d s), \quad t \in [0,T].
$$
This implies that
\begin{itemize}
\item[(i')] for any constant $c>0$, the level set $\{f; \tilde{I}(f)\le c\}$ is compact in $C([0,T];\R^{d_1+d_2})$;
\item[(ii')]for any closed subset $F\subset C([0,T];\R^{d_1+d_2})$,
$$
\limsup_{\vv\rightarrow0^+}\vv\log\P((X^\vv,\tilde{Y}^{\vv})\in F)\le -\inf_{f\in F}\tilde{I}(f);
$$
\item[(iii')] for any open subset $G\subset C([0,T];\R^{d_1+d_2})$,
$$
\liminf_{\vv\rightarrow0^+}\vv\log\P((X^\vv,\tilde{Y}^{\vv})\in G)\ge -\inf_{f\in G}\tilde{I}(f).
$$
\end{itemize}
Next, we will prove that
$\{(X_t^\vv, Y_t^\vv), t \in [0,T]\}_{\vv\in(0,1)}$ satisfies the LDP in $C([0,T],\R^{d_1+d_2})$ with the speed function $\vv^{-1}$ and the good rate function $I$ defined by
\beq\label{rate0'}
I(f):= \frac{1}{2} \inf_{f= g(h), h\in \tilde{\mathcal{H}}} \|h\|_{\tilde{H}}^2
\neqq
with
$$
(g(h))_t=\int_0^t (\bar{b}^0((g(h))_s),\bar{B}^0((g(h))_s))\d s +\int_0^t (0,\sigma((g(h))_s)\dot{h}_s\d s), \quad t \in [0,T].
$$
This will be completed in Lemma \ref{LDP'}.
\begin{lem}\label{LDP'} Assume {\bf (H1)}-{\bf (H3)}, then $\{(X_t^\vv, Y_t^\vv), t \in [0,T]\}_{\vv\in(0,1)}$ satisfies the LDP in $C([0,T],\R^{d_1+d_2})$ with the speed function $\vv^{-1}$ and the good rate function $I$ given in \eqref{rate0'}.
\end{lem}
\begin{proof}We only need to prove that (i')-(iii') hold with $\tilde{Y}$ replaced by $Y$ and the rate function $\tilde{I}$ replaced by $I$. For any $\lambda\geq\lambda_0$, $\xi=(\xi^1,\xi^2)\in C([0,T];\R^{d_1+d_2})$, let
$$(\Theta_\lambda(\xi))_t=(\xi^1_t,\theta_\lambda(\xi^2_t)),\ \ t\in[0,T].$$
Then it is easy to see that $\Theta_\lambda$ is a homeomorphism on $C([0,T];\R^{d_1+d_2})$. In fact, for any $\xi\in C([0,T];\R^{d_1+d_2})$,
$$|(\Theta_\lambda(\xi))_t-(\Theta_\lambda(\xi))_s|\leq (\|\nabla\theta_\lambda\|_\infty\vee1)|\xi_t-\xi_s|,$$
which means $\Theta_\lambda(\xi)\in C([0,T];\R^{d_1+d_2})$. Moreover, for any $\xi\in C([0,T];\R^{d_1+d_2})$, let $\eta\in C([0,T];\R^{d_1+d_2})$ be defined as $\eta_s=(\xi^1_s,\theta^{-1}_\lambda(\xi^2_s)),\ s\in[0,T]$. Then $\Theta_\lambda(\eta)=\xi$.  On the other hand, for any $\xi,\bar{\xi}\in C([0,T];\R^{d_1+d_2})$ satisfying $\Theta_\lambda(\xi)=\Theta_\lambda(\bar{\xi})$, i.e., $\xi^1_s=\bar{\xi}^1_s$ and $\theta_\lambda(\xi^2_s)=\theta_\lambda(\bar{\xi}^2_s), s\in[0,T]$, we have $\xi=\bar{\xi}$. So, $\Theta_\lambda$ is a bijection. Moreover, for any $\xi,\bar{\xi}\in C([0,T];\R^{d_1+d_2})$, we have
\begin{align*}
&\|\Theta_\lambda(\xi)-\Theta_\lambda(\bar{\xi})\|_{\infty}\leq (\|\nabla\theta_\lambda\|_\infty\vee1)\sup_{t\in[0,T]}|\xi_t-\bar{\xi}_t| =(\|\nabla\theta_\lambda\|_\infty\vee 1)\|\xi-\bar{\xi}\|_{\infty},
\end{align*}
which means that $\Theta_\lambda$ is a continuous map. Similarly, $\Theta_\lambda^{-1}$ is also a continuous map. Thus, $\Theta_\lambda$ is a homeomorphism on $C([0,T];\R^{d_1+d_2})$.

(i') We firstly prove that $I=\tilde{I}(\Theta_\lambda(\cdot))$. By chain rule and the definition of $\tilde{B}^0, \tilde{b}^0$, $\tilde{\sigma}$ and $\Theta_\lambda$, it is not difficult to see that
\begin{align*}
(\Theta_\lambda(g(h)))_t&=\int_0^t (\tilde{b}^0((\Theta_\lambda(g(h)))_s), \tilde{B}^0((\Theta_\lambda(g(h)))_s))\d s\\
 &\ \ +\int_0^t (0,\tilde{\sigma}((\Theta_\lambda(g(h)))_s)\dot{h}_s\d s), \quad t \in [0,T].
\end{align*}
By the uniqueness of solution, we have $\Theta_\lambda (g(h))=\tilde{g}(h)$. Combining the definition of $I$ and $\tilde{I}$, we arrive at $I=\tilde{I}(\Theta_\lambda(\cdot))$. Thus, for any $c>0$, $\{f; I(f)\le c\}=\{f;\tilde{I}(\Theta_\lambda(f))\leq c\}=\Theta_\lambda^{-1}\{f; \tilde{I}(f)\le c\}$. Since $\{f; \tilde{I}(f)\le c\}$ is a compact set and $\Theta_\lambda$ is a homeomorphism, we conclude that $\{f; I(f)\le c\}$ is a compact set.

(ii')  for any closed subset $F\subset C([0,T];\R^{d_1+d_2})$,
\begin{align*}
&\limsup_{\vv\rightarrow0^+}\vv\log\P((X^{\vv}, Y^\vv)\in F)\\
&=\limsup_{\vv\rightarrow0^+}\vv\log\P((X^\vv,\tilde{Y}^{\vv})\in \Theta_\lambda(F))\\
&\le -\inf_{f\in \Theta_\lambda(F)}\tilde{I}(f)\\
&= -\inf_{f\in F}\tilde{I}(\Theta_\lambda(f))=-\inf_{f\in F}I(f).
\end{align*}
Similarly, for any open subset $G\subset C([0,T];\R^{d_1+d_2})$,
$$
\liminf_{\vv\rightarrow0^+}\vv\log\P((X^{\vv}, Y^\vv)\in G)\ge -\inf_{f\in G}I(f).
$$
Thus, (iii') holds.

We finish the proof.
\end{proof}
\begin{rem}\label{Holder} By \cite[Lemma 3.2]{LLW}, we know that \eqref{u} and \eqref{u'} also hold if we assume {\bf(A2)} and {\bf(H3)} for $\phi(x)=x^\alpha$ with $\alpha\in(0,1)$. Thus, the assertions in Theorem \ref{T1.1} and Theorem \ref{T1.10} still hold by replacing \eqref{phi} and \eqref{phi'} with $\phi(x)=x^\alpha$ for some $\alpha\in(0,1)$.
\end{rem}

\paragraph{Acknowledgement.} The authors would like to thank Professor Feng-Yu Wang for helpful comments.

\beg{thebibliography}{99} {\small

\setlength{\baselineskip}{0.14in}
\parskip=0pt

\bibitem{BD} A. Budhiraja, P. Dupuis, A variational representation for positive functionals of
infinite dimensional Brownian motion, {\it Probab. Math. Statist.}, {\bf 20}(2000), 39--61.

\bibitem{BD13} A. Budhiraja, P. Dupuis, Large deviations for stochastic partial differential equations driven by a Poisson random measure, {\it Stochastic Processes $\&$ Their Applications}, {\bf 123}(2013), 523--560.

\bibitem{BDM} A. Budhiraja, P. Dupuis, V. Maroulas, Large deviations for infinite dimensional stochastic dynamical systems continuous time processes, {\it Annals of Probability},
{\bf 36}(2008), 1390--1420.

\bibitem{BGT} N. H. Bingham, C. M. Goldie, J. L. Teugels, \emph{Regular variation,} Cambridge University Press, Cambridge, 1987.

\bibitem{CF} S. Cerrai, M. Freidlin, Mark Large deviations for the Langevin equation with strong damping.
{\it J. Stat. Phys.}, {\bf 4}(2015), 859--875.

\bibitem{CR04} S. Cerrai, M. R\"ockner, Large deviations for stochastic reaction-diffusion systems
with multiplicative noise and non-Lipschitz reaction term, {\it Ann. Probab.}, {\bf 32}(2004),
1100--1139.

\bibitem{C92} P. L. Chow, Large deviation problem for some parabolic It\^o equations, {\it Commun. Pure
Appl. Math.}, {\bf 45}(1992), 97--120.

\bibitem{FZ} S. Fang, T. Zhang, A study of a class of stochastic differential equations with non-Lipschitzian coefficients, {\it Probability Theory and Related Fields}, {\bf 132}(2005), 356-390.

\bibitem{Fe-Fl}E. Fedrizzi, F. Flandoli, Pathwise uniqueness and continuous dependence for SDEs with nonirregular drift,
\emph{Stochastics} {\bf83}(2011), 241--257.

\bibitem{FW12} 	M. I. Freidlin, A. D. Wentzell, Random Perturbations of Dynamical Systems, {\it New York: Springer Science and Business Media}, vol. 260, 1984.

\bibitem{GW} F. Gao, S. Wang, Asymptotic behaviors for functionals of random dynamical systems, {\it Stoch. Anal. Appl.}, {\bf 34}(2016), 258-277.

\bibitem{GM} I. Gy\"{o}ngy, T. Martinez, On stochastic differential equations with locally unbounded drift, {\it Czechoslovak Math.J.}, {\bf51}(2001), 763--783.

\bibitem{KR} N. V. Krylov, M. R\"ockner,  Strong solutions of stochastic equations
with singular time dependent drift, {\it Probability Theory and Related
Fields}, {\bf131}(2005), 154-196.

\bibitem{KS0} A. Kulik, D. Sobolieva, Large deviation principle for one-dimensional SDEs with discontinuous coefficients, {\it Modern Stochastics: Theory and Applications}, {\bf 3}(2016), 145-164.

\bibitem{Lan} G. Lan, Large deviation principle of stochastic differential equations with non-Lipschitzian coefficients, {\it Frontiers of Mathematics in China}, {\bf 8}(2013), 1307-1321.

\bibitem{LLW} H. Li, D. Luo, J. Wang, Harnack inequalities for SDEs with multiplicative noise and non-regular drift, {\it Stoch. Dyn.} {\bf15}(2015), 18pp.

\bibitem{L10} W. Liu, Large deviations for stochastic evolution equations with small multiplicative
noise, {\it App. Math. Opt.}, {\bf 61}(2010), 27--56.

\bibitem{YR} YanLv and A.J. Roberts, Large Deviation Principle for Singularly Perturbed Stochastic Damped Wave Equations, {\it Stochastic Analysis $\&$ Applications}, {\bf 32}(2014), 50--60.

\bibitem{MSS} U. Manna, S. S. Sritharan, P. Sundar, Large deviations for the stochastic shell model
of turbulence, {\it Nonlinear Differential Equations and Applications}, {\bf 16}(2009), 493--521.

\bibitem{P94}  S. Peszat, Large deviation principle for stochastic evolution equations, {\it Probability Theory and
Related Fields.}, {\bf 98}(1994), 113--136.

\bibitem{CDR}  P. E. C. D. Raynal, Strong existence and uniqueness for stochastic differential equation with H\"older drift and degenerate noise, to appear in Ann. Inst. Henri Poincar\'e Probab. Stat..

\bibitem{RZ}  J. Ren, X. Zhang, Freidlin-Wentzell large deviations for homeomorphism flows
of non-Lipschitz SDE, {\it Bull. Sci.}, {\bf 129}(2005), 643--655.

\bibitem{RWW06} M. R\"ockner, F.-Y. Wang, L. Wu, Large deviations for stochastic Generalized
Porous Media Equations, {\it Stoch. Proc. Appl.}, {\bf 116}(2006), 1677--1689.

\bibitem{RZZ10} M. R\"ockner, T. Zhang, X. Zhang, Large deviations for stochastic tamed 3D Navier-
Stokes equations, {\it App. Math. Opt.}, {\bf 61}(2010), 267--285.

\bibitem{RS} C. Rovira, M. Sanz-sol\'e, Large deviations for stochastic Volterra equations in the plane,
{\it Potential Anal.}, {\bf 12}(2000), 359--383.

\bibitem{W}  F.-Y. Wang,  Gradient estimates and applications for SDEs in Hilbert space with multiplicative noise and Dini
continuous drift,  {\it J. Differential Equations}, {\bf260}(2016),
2792--2829.

\bibitem{WZ}F.-Y. Wang, X. Zhang, Degenerate SDE with H\"older-Dini Drift and Non-Lipschitz Noise Coefficient, {\it SIAM J. Math.
Anal.}, {\bf48}(2016), 2189--2226.

\bibitem{ZX17} T. Xu, T. Zhang, White noise driven SPDEs with reflection: Existence, uniqueness and large deviation principles, {\it Stochastic Processes $\&$ Their Applications}, {\bf 119}(2017), 3453--3470.

\bibitem{Z} X. Zhang, Strong solutions of SDEs with singural drift and Sobolev diffusion coefficients, {\it Stoch. Proc. Appl.}, {\bf115}(2005), 1805--1818.


\bibitem{Z17} X. Zhang, Stochastic Volterra equations in Banach spaces and stochastic partial differential equation, {\it Journal of Functional Analysis}, {\bf 258}(2017), 1361--1425.

\bibitem{AZ} A. K. Zvonkin, A transformation of the phase space of a diffusion process that removes the drift,  {\it Math. Sb. },
 {\bf 93}(1974), 129--149.

}
\end{thebibliography}

\end{document}